\documentclass{my_aims}

\usepackage{url}

\usepackage{amsmath}
\usepackage{paralist}
\usepackage{graphics}
\usepackage{epsfig}
\usepackage[colorlinks=true]{hyperref}

\hypersetup{urlcolor=blue, citecolor=red}

\def\currentvolume{X}

    \def\ppages{X--XX}

\newtheorem{theorem}{Theorem}[section]
\newtheorem{corollary}{Corollary}

\newtheorem{lemma}[theorem]{Lemma}

\theoremstyle{definition}
\newtheorem{definition}[theorem]{Definition}
\newtheorem{remark}{Remark}
\newtheorem{example}{Example}

\title[Fractional Difference Problems of the Calculus of Variations]{Necessary
Optimality Conditions\\
for Fractional Difference Problems\\
of the Calculus of Variations}

\author[N. R. O. Bastos, R. A. C. Ferreira and D. F. M. Torres]{}



\email{nbastos@estv.ipv.pt}
\email{ruiacferreira@ua.pt}
\email{delfim@ua.pt}

\begin{document}

\maketitle

\centerline{\scshape Nuno R. O. Bastos}
\medskip
{\footnotesize
 \centerline{Department of Mathematics, ESTGV}
   \centerline{Polytechnic Institute of Viseu}
   \centerline{3504-510 Viseu, Portugal}
}

\medskip

\centerline{\scshape Rui A. C. Ferreira}
\medskip
 {\footnotesize
 \centerline{Faculty of Engineering and Natural Sciences}
   \centerline{Lusophone University of Humanities and Technologies}
   \centerline{1749-024 Lisbon, Portugal}
}

\medskip
\centerline{\scshape Delfim F. M. Torres}
\medskip
{\footnotesize
 \centerline{Department of Mathematics}
   \centerline{University of Aveiro}
   \centerline{3810-193 Aveiro, Portugal}
}

\bigskip

\centerline{(Submitted 09-May-2009; revised 27-March-2010; accepted 04-July-2010)}

\begin{abstract}
We introduce a discrete-time fractional calculus of variations.
First and second order necessary optimality conditions are
established. Examples illustrating the use of the new Euler-Lagrange
and Legendre type conditions are given. They show that the solutions
of the fractional problems coincide with the solutions of the
corresponding non-fractional variational problems when the order of
the discrete derivatives is an integer value.
\end{abstract}


\section{Introduction}
\label{int}

The Fractional Calculus is currently a very important research
field in several different areas: physics (including classical and quantum mechanics
and thermodynamics), chemistry, biology, economics and control theory
\cite{K:S:T:06,Miller1,Ortigueira,S:K:M,TenreiroMachado}.
It has origin more than 300 years ago when L'Hopital
asked Leibniz what should be the meaning of a derivative
of order $1/2$. After that episode several more famous mathematicians
contributed to the development of Fractional Calculus:
Abel, Fourier, Liouville, Riemann, Riesz,
just to mention a few names.

In \cite{Miller} Miller and Ross define a fractional sum of
order $\nu>0$ \emph{via} the solution of a linear difference
equation. Namely, they present it as (see Section~\ref{sec0} for
the notations used here)
\begin{equation}\label{naosei8}
\Delta^{-\nu}f(t)=\frac{1}{\Gamma(\nu)}\sum_{s=a}^{t-\nu}(t-\sigma(s))^{(\nu-1)}f(s).
\end{equation}
This was done in analogy with the Riemann--Liouville fractional
integral of order $\nu>0$,
\begin{equation*}
_a\mathbf{D}_x^{-\nu}f(x)=\frac{1}{\Gamma(\nu)}\int_{a}^{x}(x-s)^{\nu-1}f(s)ds,
\end{equation*}
which can be obtained \emph{via} the solution of a linear
differential equation \cite{Miller,Miller1}. Some basic properties
of the sum in \eqref{naosei8} were obtained in \cite{Miller}.
More recently, F.~Atici and P.~Eloe \cite{Atici0,Atici}
defined the fractional difference of order $\alpha>0$, \textrm{i.e.},
$\Delta^\alpha f(t)=\Delta^m(\Delta^{-(m-\alpha)}f(t))$
with $m$ the least integer satisfying $m \ge \alpha$,
and developed some of its properties that allow
to obtain solutions of certain
fractional difference equations.

Fractional differential calculus has been widely developed in the
past few decades due mainly to its demonstrated applications in
various fields of science and engineering.
The study of fractional problems of the Calculus of Variations
and respective Euler-Lagrange equations is a fairly recent issue
-- see \cite{agr0,agr1,A:B:07,A:M:T:10,A:T:09,B:D:A:09,El-Nabulsi1,F:T:07,gasta1,M:T:10}
and references therein -- and include only the continuous case.
It is well known that discrete analogues of differential equations
can be very useful in applications \cite{Zeidan1,book:DCV}.
Therefore, we consider pertinent  to start here a fractional
discrete-time theory of the calculus of variations.

Our objective is two-fold. On one hand we
proceed to develop the theory of \emph{fractional difference
calculus}, namely, we introduce the concept of left and right
fractional sum/difference (\textrm{cf.} Definitions~\ref{def0}
and \ref{def1} below) and prove some new results related to
them. On the other hand, we believe that the present
work will potentiate research not only in
the fractional calculus of variations but also in solving fractional
difference equations, specifically, fractional equations in which
left and right fractional differences appear.

Because the theory of fractional difference calculus is in its
infancy \cite{Atici0,Atici,Miller}, the paper is self contained.
We begin, in Section~\ref{sec0},
to give the definitions and results needed throughout.
In Section~\ref{sec1} we present and prove the new results;
in Section~\ref{sec2} we give some examples.
Finally, in Section~\ref{sec:conc} we mention the
main conclusions of the paper, and some possible
extensions and open questions.
Computer code done in the Computer Algebra System
\textsf{Maxima} is given in Appendix.


\section{Preliminaries}
\label{sec0}

We begin by introducing some notation used throughout. Let $a$
be an arbitrary real number and $b = k + a$
for a certain $k \in\mathbb{N}$ with $k \ge 2$.
We put $\mathbb{T}= \{a, a + 1, \ldots, b\}$,
$\mathbb{T}^\kappa=\{a, a + 1, \ldots, b-1\}$ and
$\mathbb{T}^{\kappa^2}=\{a, \ldots, b - 2\}$. Denote by
$\mathcal{F}$ the set of all real valued functions defined on
$\mathbb{T}$. Also, we will frequently write $\sigma(t)=t+1$,
$\rho(t)=t-1$ and $f^\sigma(t)=f(\sigma(t))$. The usual
conventions $\sum_{t=c}^{c-1}f(t)=0$, $c\in\mathbb{T}$, and
$\prod_{i=0}^{-1}f(i)=1$ remain valid here.

As usual, the forward difference is defined by $\Delta
f(t)=f^\sigma(t)-f(t)$. If we have a function $f$ of two
variables, $f(t,s)$, its partial (difference) derivatives are
denoted by $\Delta_t$ and $\Delta_s$, respectively. For arbitrary
$x,y\in\mathbb{R}$ define (when it makes sense)
$$x^{(y)}=\frac{\Gamma(x+1)}{\Gamma(x+1-y)},$$
where $\Gamma$ is the gamma function. The following property of
the gamma function,
\begin{equation}
\label{naosei9}
\Gamma(x+1)=x\Gamma(x),
\end{equation}
will be frequently used.

As was mentioned in Section~\ref{int}, equality \eqref{naosei8}
was introduced in \cite{Miller} as \emph{the fractional sum of
order} $\nu>0$. While reaching the proof of
Theorem~\ref{teor1} we actually ``find" the definition of left and right
fractional sum:

\begin{definition}\label{def0}
Let $f\in\mathcal{F}$. The \emph{left fractional sum}
and the \emph{right fractional sum} of order $\nu>0$ are defined,
respectively, as
\begin{equation}\label{sum1}
_a\Delta_t^{-\nu}f(t)=\frac{1}{\Gamma(\nu)}\sum_{s=a}^{t-\nu}(t-\sigma(s))^{(\nu-1)}f(s),
\end{equation}
and
\begin{equation}\label{sum2}
_t\Delta_b^{-\nu}f(t)=\frac{1}{\Gamma(\nu)}\sum_{s=t+\nu}^{b}(s-\sigma(t))^{(\nu-1)}f(s).
\end{equation}
\end{definition}

\begin{remark}
The above sums \eqref{sum1} and \eqref{sum2} are defined for
$t \in \{a+\nu, a + \nu + 1, \ldots, b + \nu\}$ and $t \in \{a-\nu, a - \nu + 1, \ldots, b- \nu\}$,
respectively, while $f(t)$ is defined for $t \in \{a, a + 1, \ldots, b\}$. Throughout we will write
\eqref{sum1} and \eqref{sum2}, respectively, in the following way:
\begin{equation*}
_a\Delta_t^{-\nu}f(t)=\frac{1}{\Gamma(\nu)}\sum_{s=a}^{t}(t+\nu-\sigma(s))^{(\nu-1)}f(s),\quad
t\in\mathbb{T},
\end{equation*}
\begin{equation*}
_t\Delta_b^{-\nu}f(t)=\frac{1}{\Gamma(\nu)}\sum_{s=t}^{b}(s+\nu-\sigma(t))^{(\nu-1)}f(s),\quad
t\in\mathbb{T}.
\end{equation*}
\end{remark}
\begin{remark}
The left fractional sum defined in
\eqref{sum1} coincides with the fractional sum defined in
\cite{Miller} (see also \eqref{naosei8}). The analogy of
\eqref{sum1} and \eqref{sum2} with the Riemann--Liouville left and
right fractional integrals of order $\nu>0$ is clear:
$$_a\mathbf{D}_x^{-\nu}f(x)=\frac{1}{\Gamma(\nu)}\int_{a}^{x}(x-s)^{\nu-1}f(s)ds,$$
$$_x\mathbf{D}_b^{-\nu}f(x)=\frac{1}{\Gamma(\nu)}\int_{x}^{b}(s-x)^{\nu-1}f(s)ds.$$
\end{remark}
It was proved in \cite{Miller} that $\lim_{\nu\rightarrow
0}{_a}\Delta_t^{-\nu}f(t)=f(t)$. We do the same for the right
fractional sum using a different method. Let $\nu>0$ be arbitrary.
Then,
\begin{align*}
{_t}\Delta_b^{-\nu}f(t)&=\frac{1}{\Gamma(\nu)}\sum_{s=t}^{b}(s+\nu-\sigma(t))^{(\nu-1)}f(s)\\
&=f(t)+\frac{1}{\Gamma(\nu)}\sum_{s=\sigma(t)}^{b}(s+\nu-\sigma(t))^{(\nu-1)}f(s)\\
&=f(t)+\sum_{s=\sigma(t)}^{b}\frac{\Gamma(s+\nu-t)}{\Gamma(\nu)\Gamma(s-t+1)}f(s)\\
&=f(t)+\sum_{s=\sigma(t)}^{b}\frac{\prod_{i=0}^{s-t-1}(\nu+i)}{\Gamma(s-t+1)}f(s).
\end{align*}
Therefore, $\lim_{\nu\rightarrow 0}{_t}\Delta_b^{-\nu}f(t)=f(t)$.
It is now natural to define
\begin{equation}\label{naosei10}
_a\Delta_t^{0}f(t)={_t}\Delta_b^{0}f(t)=f(t),
\end{equation}
which we do here, and to write
\begin{equation}\label{seila1}
{_a}\Delta_t^{-\nu}f(t)=f(t)+\frac{\nu}{\Gamma(\nu+1)}\sum_{s
=a}^{t-1}(t+\nu-\sigma(s))^{(\nu-1)}f(s),\quad
t\in\mathbb{T},\quad \nu\geq 0,
\end{equation}
\begin{equation*}
{_t}\Delta_b^{-\nu}f(t)=f(t)+\frac{\nu}{\Gamma(\nu+1)}\sum_{s
=\sigma(t)}^{b}(s+\nu-\sigma(t))^{(\nu-1)}f(s),\quad
t\in\mathbb{T},\quad \nu\geq 0.
\end{equation*}

The next theorem was proved in \cite{Atici}.
\begin{theorem}[\cite{Atici}]
\label{thm2}
Let $f\in\mathcal{F}$ and $\nu>0$. Then, the equality
\begin{equation*}
{_a}\Delta_{t}^{-\nu}\Delta
f(t)=\Delta(_a\Delta_t^{-\nu}f(t))-\frac{(t+\nu-a)^{(\nu-1)}}{\Gamma(\nu)}f(a),\quad
t\in\mathbb{T}^\kappa,
\end{equation*}
holds.
\end{theorem}
\begin{remark}\label{rem0}
It is easy to include the case $\nu=0$ in Theorem~\ref{thm2}.
Indeed, in view of \eqref{naosei9} and \eqref{naosei10}, we get
\begin{equation}\label{seila0}
{_a}\Delta_{t}^{-\nu}\Delta
f(t)=\Delta(_a\Delta_t^{-\nu}f(t))
-\frac{\nu}{\Gamma(\nu+1)}(t+\nu-a)^{(\nu-1)}f(a),\quad
t\in\mathbb{T}^\kappa,
\end{equation}
for all $\nu\geq 0$.
\end{remark}
Now, we prove the counterpart of Theorem~\ref{thm2}
for the right fractional sum.
\begin{theorem}
Let $f\in\mathcal{F}$ and $\nu\geq 0$. Then, the equality
\begin{equation}\label{naosei12}
{_t}\Delta_{\rho(b)}^{-\nu}\Delta
f(t)=\frac{\nu}{\Gamma(\nu+1)}(b+\nu-\sigma(t))^{(\nu-1)}f(b)
+\Delta(_t\Delta_b^{-\nu}f(t)),\quad t\in\mathbb{T}^\kappa,
\end{equation}
holds.
\end{theorem}

\begin{proof}
We only prove the case $\nu>0$ as the case $\nu=0$ is trivial (see
Remark~\ref{rem0}). We start by fixing an arbitrary
$t\in\mathbb{T}^\kappa$. Then, we have that, for all
$s\in\mathbb{T}^\kappa$,
\begin{multline*}
\Delta_s\left((s+\nu-\sigma(t))^{(\nu-1)}f(s))\right)\\
=(\nu-1)(s+\nu-\sigma(t))^{(\nu-2)}f^\sigma(s)
+(s+\nu-\sigma(t))^{(\nu-1)}\Delta f(s),
\end{multline*}
hence,
\begin{equation*}
\begin{split}
\frac{1}{\Gamma(\nu)} & \sum_{s=t}^{b-1}(s+\nu-\sigma(t))^{(\nu-1)}\Delta f(s)\\
&=\left[\frac{(s+\nu-\sigma(t))^{(\nu-1)}}{\Gamma(\nu)}f(s)\right]_{s=t}^{s=b}
-\frac{1}{\Gamma(\nu)}\sum_{s=t}^{b-1}(\nu-1)(s+\nu-\sigma(t))^{(\nu-2)}
f^\sigma(s)\\
&=\frac{(b+\nu-\sigma(t))^{(\nu-1)}}{\Gamma(\nu)}f(b)
-\frac{(\nu-1)^{(\nu-1)}}{\Gamma(\nu)}f(t)\\
&\qquad
-\frac{1}{\Gamma(\nu)}\sum_{s=t}^{b-1}(\nu-1)(s+\nu-\sigma(t))^{(\nu-2)}
f^\sigma(s).
\end{split}
\end{equation*}
We now compute $\Delta(_t\Delta_b^{-\nu}f(t))$:
\begin{equation*}
\begin{split}
\Delta(_t\Delta_b^{-\nu}f(t))&=
\frac{1}{\Gamma(\nu)}\left[\sum_{s=\sigma(t)}^{b}(s+\nu-\sigma(t+1))^{(\nu-1)}
f(s)\right.\\
& \left. \qquad\qquad\quad -\sum_{s=t}^{b}(s+\nu-\sigma(t))^{(\nu-1)} f(s)\right]\\
&=\frac{1}{\Gamma(\nu)}\left[\sum_{s=\sigma(t)}^{b}(s+\nu-\sigma(t+1))^{(\nu-1)}
f(s)\right.\\
&\qquad\qquad\quad \left.-\sum_{s=\sigma(t)}^{b}(s+\nu-\sigma(t))^{(\nu-1)}
f(s)\right]-\frac{(\nu-1)^{(\nu-1)}}{\Gamma(\nu)}f(t)\\
&=\frac{1}{\Gamma(\nu)}\sum_{s=\sigma(t)}^{b}\Delta_t(s+\nu-\sigma(t))^{(\nu-1)}
f(s)-\frac{(\nu-1)^{(\nu-1)}}{\Gamma(\nu)}f(t)\\
&=-\frac{1}{\Gamma(\nu)}\sum_{s=t}^{b-1}(\nu-1)(s+\nu-\sigma(t))^{(\nu-2)}
f^\sigma(s)-\frac{(\nu-1)^{(\nu-1)}}{\Gamma(\nu)}f(t).
\end{split}
\end{equation*}
Since $t$ is arbitrary, the theorem is proved.
\end{proof}

\begin{definition}\label{def1}
Let $0<\alpha\leq 1$ and set $\mu=1-\alpha$. Then, the \emph{left
fractional difference} and the \emph{right fractional difference}
of order $\alpha$ of a function $f\in\mathcal{F}$ are defined, respectively, by
$$_a\Delta_t^\alpha f(t)=\Delta(_a\Delta_t^{-\mu}f(t)),\quad t\in\mathbb{T}^\kappa,$$
and
$$_t\Delta_b^\alpha f(t)=-\Delta(_t\Delta_b^{-\mu}f(t))),\quad t\in\mathbb{T}^\kappa.$$
\end{definition}


\section{Main Results}
\label{sec1}

Our aim is to introduce the discrete-time fractional problem
of the calculus of variations and
to prove corresponding necessary optimality conditions.
In order to obtain an analogue of the Euler-Lagrange
equation (\textrm{cf.} Theorem~\ref{thm0}) we
first prove a fractional formula of summation by parts.
The results of the paper give discrete analogues to
the fractional Riemann--Liouville results available
in the literature: Theorem~\ref{teor1}
is the discrete analog of fractional
integration by parts \cite{Riewe,S:K:M};
Theorem~\ref{thm0} is the discrete analog
of the fractional Euler-Lagrange equation of Agrawal
\cite[Theorem~1]{agr0}; the natural boundary conditions
\eqref{rui1} and \eqref{rui2} are the discrete fractional analogues
of the transversality conditions in \cite{A:TC:06,M:T:10}.
However, to the best of the authors knowledge, no counterpart to
our Theorem~\ref{thm1} exists in the literature of
continuous fractional variational problems.


\subsection{Fractional Summation by Parts}

The next lemma is used in the proof of Theorem~\ref{teor1}.

\begin{lemma}\label{lem0}
Let $f$ and $h$ be two functions defined on $\mathbb{T}^\kappa$
and $g$ a function defined on
$\mathbb{T}^\kappa\times\mathbb{T}^\kappa$. Then,
the equality
\begin{equation*}
\sum_{\tau=a}^{b-1}f(\tau)\sum_{s=a}^{\tau-1}g(\tau,s)h(s)
=\sum_{\tau=a}^{b-2}h(\tau)\sum_{s=\sigma(\tau)}^{b-1}g(s,\tau)f(s)
\end{equation*}
holds.
\end{lemma}
\begin{proof}
Choose $\mathbb{T} = \mathbb{Z}$ and
$F(\tau,s) = f(\tau) g(\tau,s) h(s)$ in Theorem~10 of \cite{Akin}.
\end{proof}

The next result gives a \emph{fractional summation by parts} formula.

\begin{theorem}[Fractional summation by parts]
\label{teor1}
Let $f$ and $g$ be real valued functions defined on $\mathbb{T}^k$
and $\mathbb{T}$, respectively. Fix $0<\alpha\leq 1$ and put
$\mu=1-\alpha$. Then,
\begin{multline*}
\sum_{t=a}^{b-1}f(t)_a\Delta_t^\alpha
g(t)=f(b-1)g(b)-f(a)g(a)+\sum_{t=a}^{b-2}{_t\Delta_{\rho(b)}^\alpha
f(t)g^\sigma(t)}\\
+\frac{\mu}{\Gamma(\mu+1)}g(a)\left(\sum_{t=a}^{b-1}(t+\mu-a)^{(\mu-1)}f(t)
-\sum_{t=\sigma(a)}^{b-1}(t+\mu-\sigma(a))^{(\mu-1)}f(t)\right).
\end{multline*}
\end{theorem}

\begin{proof}
From \eqref{seila0} we can write
\begin{align}
\sum_{t=a}^{b-1}f(t)_a\Delta_t^\alpha
g(t)&=\sum_{t=a}^{b-1}f(t)\Delta(_a\Delta_t^{-\mu} g(t))\nonumber\\
&=\sum_{t=a}^{b-1}f(t)\left[_a\Delta_t^{-\mu}\Delta
g(t)+\frac{\mu}{\Gamma(\mu+1)}(t+\mu-a)^{(\mu-1)}g(a)\right]\nonumber\\
&=\sum_{t=a}^{b-1}f(t)_a\Delta_t^{-\mu}\Delta
g(t)+\sum_{t=a}^{b-1}\frac{\mu}{\Gamma(\mu+1)}(t+\mu-a)^{(\mu-1)}f(t)g(a)\label{rui0}.
\end{align}
Using \eqref{seila1} we get
\begin{align*}
\sum_{t=a}^{b-1}&f(t)_a\Delta_t^{-\mu}\Delta
g(t)\\
&=\sum_{t=a}^{b-1}f(t)\Delta
g(t)+\frac{\mu}{\Gamma(\mu+1)}\sum_{t=a}^{b-1}f(t)\sum_{s=a}^{t-1}(t+\mu-\sigma(s))^{(\mu-1)}\Delta
g(s)\\
&=\sum_{t=a}^{b-1}f(t)\Delta
g(t)+\frac{\mu}{\Gamma(\mu+1)}\sum_{t=a}^{b-2}\Delta
g(t)\sum_{s=\sigma(t)}^{b-1}(s+\mu-\sigma(t))^{(\mu-1)}f(s)\\
&=f(b-1)[g(b)-g(b-1)]+\sum_{t=a}^{b-2}\Delta
g(t)_t\Delta_{\rho(b)}^{-\mu} f(t),
\end{align*}
where the third equality follows by Lemma~\ref{lem0}. We proceed
to develop the right hand side of the last equality as follows:
\begin{equation*}
\begin{split}
f&(b-1)[g(b)-g(b-1)]+\sum_{t=a}^{b-2}\Delta
g(t)_t\Delta_{\rho(b)}^{-\mu} f(t)\\
&=f(b-1)[g(b)-g(b-1)] +\left[g(t)_t\Delta_{\rho(b)}^{-\mu}
f(t)\right]_{t=a}^{t=b-1}-\sum_{t=a}^{b-2}
g^\sigma(t)\Delta(_t\Delta_{\rho(b)}^{-\mu} f(t))\\
&=f(b-1)g(b)-f(a)g(a)-\frac{\mu}{\Gamma(\mu+1)}g(a)\sum_{s
=\sigma(a)}^{b-1}(s+\mu-\sigma(a))^{(\mu-1)}f(s)\\
&\qquad +\sum_{t=a}^{b-2}{\left(_t\Delta_{\rho(b)}^\alpha
f(t)\right)g^\sigma(t)}\, ,
\end{split}
\end{equation*}
where the first equality follows from the usual summation by parts
formula. Putting this into \eqref{rui0}, we get:
\begin{multline*}
\sum_{t=a}^{b-1}f(t)_a\Delta_t^\alpha
g(t)=f(b-1)g(b)-f(a)g(a)+\sum_{t=a}^{b-2}{\left(_t\Delta_{\rho(b)}^\alpha
f(t)\right)g^\sigma(t)}\\
+\frac{g(a)\mu}{\Gamma(\mu+1)}\sum_{t=a}^{b-1}\frac{(t+\mu-a)^{(\mu-1)}}{\Gamma(\mu)}f(t)
-\frac{g(a)\mu}{\Gamma(\mu+1)}\sum_{s=\sigma(a)}^{b-1}(s+\mu-\sigma(a))^{(\mu-1)}f(s).
\end{multline*}
The theorem is proved.
\end{proof}


\subsection{Necessary Optimality Conditions}

We begin to fix two arbitrary real numbers $\alpha$ and $\beta$
such that $\alpha,\beta\in(0,1]$. Further, we put $\mu=1-\alpha$
and $\nu=1-\beta$.

Let a function $L(t,u,v,w):\mathbb{T}^\kappa\times\mathbb{R}\times
\mathbb{R}\times\mathbb{R}\rightarrow\mathbb{R}$
be given. We assume that the second-order partial derivatives
$L_{uu}$, $L_{uv}$, $L_{uw}$, $L_{vw}$,
$L_{vv}$, and $L_{ww}$ exist and are continuous.

Consider the functional
$\mathcal{L}:\mathcal{F}\rightarrow\mathbb{R}$ defined by
\begin{equation}
\label{naosei7}
\mathcal{L}(y(\cdot))=\sum_{t=a}^{b-1}L(t,y^{\sigma}(t),{_a}\Delta_t^\alpha
y(t),{_t}\Delta_b^\beta y(t))
\end{equation}
and the problem, that we denote by (P), of minimizing
\eqref{naosei7} subject to the boundary conditions $y(a)=A$ and
$y(b)=B$ ($A,B\in\mathbb{R}$). Our aim is to derive necessary
conditions of first and second order for problem (P).
\begin{definition}
\label{def:norm}
For $f\in\mathcal{F}$ we define the norm
$$\|f\|=\max_{t\in\mathbb{T}^\kappa}|f^\sigma(t)|
+\max_{t\in\mathbb{T}^\kappa}|_a\Delta_t^\alpha
f(t)|+\max_{t\in\mathbb{T}^\kappa}|_t\Delta_b^\beta f(t)|.$$
A function $\tilde{y}\in\mathcal{F}$ with $\tilde{y}(a)=A$
and $\tilde{y}(b)=B$ is called a local minimizer for problem (P)
provided there exists $\delta>0$ such that
$\mathcal{L}(\tilde{y})\leq\mathcal{L}(y)$ for all $y\in\mathcal{F}$
with $y(a)=A$ and $y(b)=B$ and $\|y-\tilde{y}\|<\delta$.
\end{definition}

\begin{remark}
It is easy to see that Definition~\ref{def:norm}
gives a norm in $\mathcal{F}$. Indeed,
it is clear that $||f||$ is nonnegative,
and for an arbitrary $f\in\mathcal{F}$ and
$k\in\mathbb{R}$ we have $\|kf\|=|k|\|f\|$.
The triangle inequality is also easy to prove:
\begin{align*}
\|f+g\|&=\max_{t\in\mathbb{T}^\kappa}|f(t)+g(t)|
+\max_{t\in\mathbb{T}^\kappa}|_a\Delta_t^\alpha (f+g)(t)|
+\max_{t\in\mathbb{T}^\kappa}|_t\Delta_b^\alpha(f+g)(t)|\\
&\leq\max_{t\in\mathbb{T}^\kappa}\left[|f(t)|+|g(t)|\right]
+\max_{t\in\mathbb{T}^\kappa}\left[|_a\Delta_t^\alpha
f(t)|+|_a\Delta_t^\alpha
g(t)|\right]\\
& \qquad +\max_{t\in\mathbb{T}^\kappa}\left[|_t\Delta_b^\alpha
f(t)|+|_t\Delta_b^\alpha g(t)|\right]\\
&\leq\|f\|+\|g\|.
\end{align*}
The only possible doubt is
to prove that $||f|| = 0$ implies that $f(t) = 0$
for any $t \in \mathbb{T} = \{a, a+1,\ldots,b\}$.
Suppose $||f|| = 0$. It follows that
\begin{gather}
\max_{t\in\mathbb{T}^\kappa}|f^\sigma(t)|  = 0\, , \label{rr2:1}\\
\max_{t\in\mathbb{T}^\kappa}|_a\Delta_t^\alpha f(t)| = 0\, ,\label{rr2:2}\\
\max_{t\in\mathbb{T}^\kappa}|_t\Delta_b^\beta f(t)| = 0 \, . \label{rr2:3}
\end{gather}
From \eqref{rr2:1} we conclude that
$f(t) = 0$ for all  $t \in \{a+1,\ldots,b\}$.
It remains to prove that $f(a) = 0$.
To prove this we use \eqref{rr2:2} (or \eqref{rr2:3}).
Indeed, from \eqref{rr2:1} we can write
\begin{equation*}
\begin{split}
_a\Delta_t^\alpha f(t)
&=  \Delta\left(\frac{1}{\Gamma(1-\alpha)}\sum_{s=a}^{t}(t+1-\alpha-\sigma(s))^{(-\alpha)} f(s) \right)\\
&= \frac{1}{\Gamma(1-\alpha)}\left(\sum_{s=a}^{t+1}(t+2 - \alpha-\sigma(s))^{(-\alpha)} f(s)
- \sum_{s=a}^{t}(t+1-\alpha-\sigma(s))^{(-\alpha)} f(s) \right)\\
&= \frac{1}{\Gamma(1-\alpha)} \left( (t+2 - \alpha-\sigma(a))^{(-\alpha)} f(a)
- (t+1 - \alpha-\sigma(a))^{(-\alpha)} f(a)\right)\\
&= \frac{f(a)}{\Gamma(1-\alpha)} \Delta (t+1 - \alpha-\sigma(a))^{(-\alpha)}
\end{split}
\end{equation*}
and since by \eqref{rr2:2}
$_a\Delta_t^\alpha f(t) = 0$, one concludes that
$f(a) = 0$ (because $(t+1 - \alpha-\sigma(a))^{(-\alpha)}$
is not a constant).
\end{remark}

\begin{definition}
A function $\eta\in\mathcal{F}$ is called an admissible variation
for problem (P) provided $\eta\neq 0$ and $\eta(a)=\eta(b)=0$.
\end{definition}

The next theorem presents a first order necessary condition for
problem (P).

\begin{theorem}[The fractional discrete-time Euler--Lagrange equation]
\label{thm0}
If $\tilde{y}\in\mathcal{F}$ is a local minimizer
for problem (P), then
\begin{equation}\label{EL}
L_u[\tilde{y}](t) +{_t}\Delta_{\rho(b)}^\alpha L_v[\tilde{y}](t)+{_a}\Delta_t^\beta
L_w[\tilde{y}](t)=0
\end{equation}
holds for all $t\in\mathbb{T}^{\kappa^2}$,
where the operator $[\cdot]$ is defined by
$$
[y](s) =(s,y^{\sigma}(s),{_a}\Delta_s^\alpha
y(s),{_s}\Delta_b^\beta y(s)) .
$$
\end{theorem}

\begin{proof}
Suppose that $\tilde{y}(\cdot)$ is a local minimizer of
$\mathcal{L}(\cdot)$. Let $\eta(\cdot)$ be an arbitrary fixed
admissible variation and define the function
$\Phi:\left(-\frac{\delta}{\|\eta(\cdot)\|},
\frac{\delta}{\|\eta(\cdot)\|}\right)\rightarrow\mathbb{R}$
by
\begin{equation}
\label{fi}
\Phi(\varepsilon)=\mathcal{L}(\tilde{y}(\cdot)+\varepsilon\eta(\cdot)).
\end{equation}
This function has a minimum at $\varepsilon=0$, so we must
have $\Phi'(0)=0$, \textrm{i.e.},
$$\sum_{t=a}^{b-1}\left[L_u[\tilde{y}](t)\eta^\sigma(t)
+L_v[\tilde{y}](t){_a}\Delta_t^\alpha\eta(t)
+L_w[\tilde{y}](t){_t}\Delta_b^\beta\eta(t)\right]=0,$$ which we may
write, equivalently, as
\begin{multline}
\label{rui3}
L_u[\tilde{y}](t)\eta^\sigma(t)|_{t=\rho(b)}
+\sum_{t=a}^{b-2}L_u[\tilde{y}](t)\eta^\sigma(t)\\
+\sum_{t=a}^{b-1}L_v[\tilde{y}](t){_a}\Delta_t^\alpha\eta(t)
+\sum_{t=a}^{b-1}L_w[\tilde{y}](t){_t}\Delta_b^\beta\eta(t)=0.
\end{multline}
Using Theorem~\ref{teor1}, and the fact that $\eta(a)=\eta(b)=0$,
we get for the third term in \eqref{rui3} that
\begin{equation}
\label{naosei5}
\sum_{t=a}^{b-1}L_v[\tilde{y}](t){_a}\Delta_t^\alpha\eta(t)
=\sum_{t=a}^{b-2}\left({_t}\Delta_{\rho(b)}^\alpha
L_v[\tilde{y}](t)\right)\eta^\sigma(t).
\end{equation}
Using \eqref{naosei12} it follows that
\begin{equation}
\label{naosei4}
\begin{aligned}
\sum_{t=a}^{b-1}&L_w[\tilde{y}](t){_t}\Delta_b^\beta\eta(t)\\
&=-\sum_{t=a}^{b-1}L_w[\tilde{y}](t)\Delta({_t}\Delta_b^{-\nu}\eta(t))\\
&=-\sum_{t=a}^{b-1}L_w[\tilde{y}](t)\left[{_t}\Delta_{\rho(b)}^{-\nu}\Delta
\eta(t)-\frac{\nu}{\Gamma(\nu+1)}(b+\nu-\sigma(t))^{(\nu-1)}\eta(b)\right]\\
&=-\left(\sum_{t=a}^{b-1}L_w[\tilde{y}](t){_t}\Delta_{\rho(b)}^{-\nu}\Delta
\eta(t)-\frac{\nu\eta(b)}{\Gamma(\nu+1)}\sum_{t=a}^{b-1}(b
+\nu-\sigma(t))^{(\nu-1)}L_w[\tilde{y}](t)\right).
\end{aligned}
\end{equation}
We now use Lemma~\ref{lem0} to get
\begin{equation}
\label{naosei2}
\begin{split}
\sum_{t=a}^{b-1} & L_w[\tilde{y}](t){_t}\Delta_{\rho(b)}^{-\nu}\Delta
\eta(t)\\
&=\sum_{t=a}^{b-1}L_w[\tilde{y}](t)\Delta\eta(t)+\frac{\nu}{\Gamma(\nu+1)}\sum_{t=a}^{b-2}
L_w[\tilde{y}](t)\sum_{s=\sigma(t)}^{b-1}(s+\nu-\sigma(t))^{(\nu-1)}
\Delta\eta(s)\\
&=\sum_{t=a}^{b-1}L_w[\tilde{y}](t)\Delta\eta(t)
+\frac{\nu}{\Gamma(\nu+1)}\sum_{t=a}^{b-1}\Delta\eta(t)\sum_{s=a}^{t
-1}(t+\nu-\sigma(s))^{(\nu-1)}L_w[\tilde{y}](s)\\
&=\sum_{t=a}^{b-1}\Delta\eta(t){_a}\Delta^{-\nu}_t
L_w[\tilde{y}](t).
\end{split}
\end{equation}
We apply again the usual summation by parts formula,
this time to \eqref{naosei2}, to obtain:
\begin{equation}
\label{naosei3}
\begin{split}
\sum_{t=a}^{b-1} & \Delta\eta(t){_a}\Delta^{-\nu}_t
L_w[\tilde{y}](t)\\
&=\sum_{t=a}^{b-2}\Delta\eta(t){_a}\Delta^{-\nu}_t
L_w[\tilde{y}](t)+(\eta(b)-\eta(\rho(b))){_a}\Delta^{-\nu}_t
L_w[\tilde{y}](t)|_{t=\rho(b)}\\
&=\left[\eta(t){_a}\Delta^{-\nu}_t
L_w[\tilde{y}](t)\right]_{t=a}^{t=b-1}-\sum_{t=a}^{b-2}\eta^\sigma(t)\Delta({_a}\Delta^{-\nu}_t
L_w[\tilde{y}](t))\\
&\qquad +\eta(b){_a}\Delta^{-\nu}_t
L_w[\tilde{y}](t)|_{t=\rho(b)}-\eta(b-1){_a}\Delta^{-\nu}_t
L_w[\tilde{y}](t)|_{t=\rho(b)}\\
&=\eta(b){_a}\Delta^{-\nu}_t
L_w[\tilde{y}](t)|_{t=\rho(b)}-\eta(a){_a}\Delta^{-\nu}_t
L_w[\tilde{y}](t)|_{t=a}-\sum_{t=a}^{b-2}\eta^\sigma(t){_a}\Delta^{\beta}_t
L_w[\tilde{y}](t).
\end{split}
\end{equation}
Since $\eta(a)=\eta(b)=0$ it follows, from \eqref{naosei2} and
\eqref{naosei3}, that
$$
\sum_{t=a}^{b-1}L_w[\tilde{y}](t){_t}\Delta_{\rho(b)}^{-\nu}\Delta
\eta(t)=-\sum_{t=a}^{b-2}\eta^\sigma(t){_a}\Delta^{\beta}_t
L_w[\tilde{y}](t)
$$
and, after inserting in \eqref{naosei4}, that
\begin{equation}\label{naosei6}
\sum_{t=a}^{b-1}L_w[\tilde{y}](t){_t}\Delta_b^\beta\eta(t)
=\sum_{t=a}^{b-2}\eta^\sigma(t){_a}\Delta^{\beta}_t
L_w[\tilde{y}](t).
\end{equation}
By \eqref{naosei5} and \eqref{naosei6}
we may write \eqref{rui3} as
$$
\sum_{t=a}^{b-2}\left[L_u[\tilde{y}](t)
+{_t}\Delta_{\rho(b)}^\alpha L_v[\tilde{y}](t)+{_a}\Delta_t^\beta
L_w[\tilde{y}](t)\right]\eta^\sigma(t)=0.
$$
Since the values of $\eta^\sigma(t)$ are arbitrary for
$t\in\mathbb{T}^{\kappa^2}$, the Euler--Lagrange equation
\eqref{EL} holds along $\tilde{y}$.
\end{proof}

\begin{remark}
If the initial condition $y(a)=A$ is not present (\textrm{i.e.}, $y(a)$ is free),
we can use standard techniques to show that the following
supplementary condition must be fulfilled:
\begin{multline}\label{rui1}
-L_v(a)+\frac{\mu}{\Gamma(\mu+1)}\left(\sum_{t=a}^{b-1}(t+\mu-a)^{(\mu-1)}L_v[\tilde{y}](t)\right.\\
\left.-\sum_{t=\sigma(a)}^{b-1}(t+\mu-\sigma(a))^{(\mu-1)}L_v[\tilde{y}](t)\right)+
L_w(a)=0.
\end{multline}
Similarly, if $y(b)=B$ is not present (\textrm{i.e.}, $y(b)$ is free), the
equality
\begin{multline}\label{rui2}
L_u(\rho(b))+L_v(\rho(b))-L_w(\rho(b))\\
+\frac{\nu}{\Gamma(\nu+1)}\Biggl(\sum_{t=a}^{b-1}(b+\nu-\sigma(t))^{(\nu-1)}L_w[\tilde{y}](t)\\
-\sum_{t=a}^{b-2}(\rho(b)+\nu-\sigma(t))^{(\nu-1)}L_w[\tilde{y}](t)\Biggr)=0
\end{multline}
holds. We just note that the first term in \eqref{rui2} arises
from the first term on the left hand side of \eqref{rui3}.
Equalities \eqref{rui1} and \eqref{rui2} are the fractional
discrete-time \emph{natural boundary conditions}.
\end{remark}
The next result is a particular case of our Theorem~\ref{thm0}.

\begin{corollary}[The discrete-time Euler--Lagrange equation
-- \textrm{cf.}, \textrm{e.g.}, \cite{CD:Bohner:2004,RD}]
If $\tilde{y}$ is a solution to the problem
\begin{equation}
\label{eq:BPCV:DT}
\begin{gathered}
\mathcal{L}(y(\cdot))=\sum_{t=a}^{b-1}L(t,y(t+1),\Delta y(t)) \longrightarrow \min \\
y(a)=A \, , \quad y(b)=B \, ,
\end{gathered}
\end{equation}
then
$L_u(t,\tilde{y}(t+1),\Delta\tilde{y}(t))
-\Delta L_v(t,\tilde{y}(t+1),\Delta\tilde{y}(t))=0$
for all $t \in \{ a, \ldots, b-2\}$.
\end{corollary}
\begin{proof}
Follows from Theorem~\ref{thm0} with $\alpha=1$
and a $L$ not depending on $w$.
\end{proof}

We derive now the second order necessary condition for problem
(P), \textrm{i.e.}, we obtain Legendre's necessary condition for the
fractional difference setting.
\begin{theorem}[The fractional discrete-time Legendre condition]
\label{thm1}
If $\tilde{y}\in\mathcal{F}$ is a local minimizer for problem
(P), then the inequality
\begin{multline*}
L_{uu}[\tilde{y}](t)+2L_{uv}[\tilde{y}](t)
+L_{vv}[\tilde{y}](t)+L_{vv}[\tilde{y}](\sigma(t))(\mu-1)^2\\
+\sum_{s=\sigma(\sigma(t))}^{b-1}L_{vv}[\tilde{y}](s)\left(
\frac{\mu(\mu-1)\prod_{i=0}^{s-t-3}(\mu+i+1)}{(s-t)\Gamma(s-t)}\right)^2
+2L_{uw}[\tilde{y}](t)(\nu-1)\\
+2(\nu-1)L_{vw}[\tilde{y}](t)
+2(\mu-1)L_{vw}[\tilde{y}](\sigma(t))+L_{ww}[\tilde{y}](t)(1-\nu)^2\\
+L_{ww}[\tilde{y}](\sigma(t))
+\sum_{s=a}^{t-1}L_{ww}[\tilde{y}](s)\left(\frac{\nu(1-\nu)\prod_{i=0}^{t
-s-2}(\nu+i)}{(\sigma(t)-s)\Gamma(\sigma(t)-s)}\right)^2
\geq 0
\end{multline*}
holds for all $t\in\mathbb{T}^{\kappa^2}$, where
$[\tilde{y}](t)=(t,\tilde{y}^{\sigma}(t),{_a}\Delta_t^\alpha
\tilde{y}(t),{_t}\Delta_b^\beta\tilde{y}(t))$.
\end{theorem}

\begin{proof}
By the hypothesis of the theorem, and letting $\Phi$ be as in
\eqref{fi}, we get
\begin{equation}
\label{des0}
\Phi''(0)\geq 0
\end{equation}
for an arbitrary admissible variation $\eta(\cdot)$.
Inequality \eqref{des0} is equivalent to
\begin{multline*}
\sum_{t=a}^{b-1}\left[L_{uu}[\tilde{y}](t)(\eta^\sigma(t))^2
+2L_{uv}[\tilde{y}](t)\eta^\sigma(t){_a}\Delta_t^\alpha\eta(t)
+L_{vv}[\tilde{y}](t)({_a}\Delta_t^\alpha\eta(t))^2\right.\\
\left.+2L_{uw}[\tilde{y}](t)\eta^\sigma(t){_t}\Delta_b^\beta\eta(t)
+2L_{vw}[\tilde{y}](t){_a}\Delta_t^\alpha\eta(t){_t}\Delta_b^\beta\eta(t)
+L_{ww}[\tilde{y}](t)({_t}\Delta_b^\beta\eta(t))^2\right]\geq 0.
\end{multline*}
Let $\tau\in\mathbb{T}^{\kappa^2}$ be arbitrary and define
$\eta:\mathbb{T}\rightarrow\mathbb{R}$ by
\[ \eta(t) = \left\{ \begin{array}{ll}
1 & \mbox{if $t=\sigma(\tau)$};\\
0 & \mbox{otherwise}.\end{array} \right. \]
It follows that
$\eta(a)=\eta(b)=0$, \textrm{i.e.},
$\eta$ is an admissible variation.
Using \eqref{seila0} (note that $\eta(a)=0$), we get
\begin{equation*}
\begin{split}
\sum_{t=a}^{b-1}&\left[L_{uu}[\tilde{y}](t)(\eta^\sigma(t))^2
+2L_{uv}[\tilde{y}](t)\eta^\sigma(t){_a}\Delta_t^\alpha\eta(t)
+L_{vv}[\tilde{y}](t)({_a}\Delta_t^\alpha\eta(t))^2\right]\\
&=\sum_{t=a}^{b-1}\left\{L_{uu}[\tilde{y}](t)(\eta^\sigma(t))^2\right.\\
&\qquad \left. +2L_{uv}[\tilde{y}](t)\eta^\sigma(t)\left[\Delta\eta(t)
+\frac{\mu}{\Gamma(\mu+1)}\sum_{s=a}^{t-1}(t
+\mu-\sigma(s))^{(\mu-1)}\Delta\eta(s)\right]\right.\\
&\qquad \left. +L_{vv}[\tilde{y}](t)\left(\Delta\eta(t)
+\frac{\mu}{\Gamma(\mu+1)}\sum_{s=a}^{t-1}(t
+\mu-\sigma(s))^{(\mu-1)}\Delta\eta(s)\right)^2\right\}\\
&= L_{uu}[\tilde{y}](\tau)+2L_{uv}[\tilde{y}](\tau)+L_{vv}[\tilde{y}](\tau)\\
&\qquad +\sum_{t=\sigma(\tau)}^{b-1}L_{vv}[\tilde{y}](t)\left(\Delta\eta(t)
+\frac{\mu}{\Gamma(\mu+1)}\sum_{s=a}^{t-1}(t+\mu-\sigma(s))^{(\mu-1)}\Delta\eta(s)\right)^2.
\end{split}
\end{equation*}
Observe that
\begin{multline*}
\sum_{t=\sigma(\sigma(\tau))}^{b
-1}L_{vv}[\tilde{y}](t)\left(\frac{\mu}{\Gamma(\mu+1)}\sum_{s=a}^{t-1}(t
+\mu-\sigma(s))^{(\mu-1)}\Delta\eta(s)\right)^2
+ L_{vv}(\sigma(\tau))(-1+\mu)^2 \\
= \sum_{t=\sigma(\tau)}^{b-1}L_{vv}[\tilde{y}](t)\left(\Delta\eta(t)
+\frac{\mu}{\Gamma(\mu+1)}\sum_{s=a}^{t-1}(t+\mu-\sigma(s))^{(\mu-1)}\Delta\eta(s)\right)^2
\, .
\end{multline*}
We show next that
\begin{multline*}
\sum_{t=\sigma(\sigma(\tau))}^{b-1}L_{vv}[\tilde{y}](t)\left(\frac{\mu}{\Gamma(\mu
+1)}\sum_{s=a}^{t-1}(t+\mu-\sigma(s))^{(\mu-1)}\Delta\eta(s)\right)^2\\
=\sum_{t=\sigma(\sigma(\tau))}^{b-1}L_{vv}[\tilde{y}](t)\left(\frac{\mu(\mu-1)\prod_{i=0}^{t
-\tau-3}(\mu+i+1)}{(t-\tau)\Gamma(t-\tau)}\right)^2.
\end{multline*}
Let $t\in[\sigma(\sigma(\tau)),b-1]\cap\mathbb{Z}$. Then,
\begin{equation*}
\begin{split}
&\frac{\mu}{\Gamma(\mu+1)}\sum_{s=a}^{t-1}(t+\mu-\sigma(s))^{(\mu-1)}\Delta\eta(s)\\
&\quad =\frac{\mu}{\Gamma(\mu+1)}\left[\sum_{s=a}^{\tau}(t+\mu-\sigma(s))^{(\mu-1)}\Delta\eta(s)
+\sum_{s=\sigma(\tau)}^{t-1}(t+\mu-\sigma(s))^{(\mu-1)}\Delta\eta(s)\right]\\
&\quad =\frac{\mu}{\Gamma(\mu+1)}\left[(t+\mu-\sigma(\tau))^{(\mu-1)}-(t+\mu-\sigma(\sigma(\tau)))^{(\mu-1)}\right]\\
&\quad=\frac{\mu}{\Gamma(\mu+1)}\left[\frac{\Gamma(t+\mu-\sigma(\tau)+1)}{\Gamma(t
+\mu-\sigma(\tau)+1-(\mu-1))}-\frac{\Gamma(t+\mu-\sigma(\sigma(\tau))
+1)}{\Gamma(t+\mu-\sigma(\sigma(\tau))+1-(\mu-1))}\right]
\end{split}
\end{equation*}
\begin{equation}
\label{rui10}
\begin{split}
&\quad=\frac{\mu}{\Gamma(\mu+1)}\left[\frac{\Gamma(t+\mu-\tau)}{\Gamma(t-\tau+1)}
-\frac{\Gamma(t-\tau+\mu-1)}{\Gamma(t-\tau)}\right]\\
&\quad=\frac{\mu}{\Gamma(\mu+1)}\left[\frac{(t+\mu-\tau-1)\Gamma(t+\mu-\tau-1)}{(t
-\tau)\Gamma(t-\tau)}-\frac{(t-\tau)\Gamma(t-\tau+\mu-1)}{(t-\tau)\Gamma(t-\tau)}\right]\\
&\quad=\frac{\mu}{\Gamma(\mu+1)}\frac{(\mu-1)\Gamma(t-\tau+\mu-1)}{(t-\tau)\Gamma(t-\tau)}\\
&\quad=\frac{\mu(\mu-1)\prod_{i=0}^{t-\tau-3}(\mu+i+1)}{(t-\tau)\Gamma(t-\tau)},
\end{split}
\end{equation}
which proves our claim. Observe that we can write
${_t}\Delta_b^\beta\eta(t)=-{_t}\Delta_{\rho(b)}^{-\nu}\Delta
\eta(t)$ since $\eta(b)=0$. It is not difficult to see that the following equality holds:
$$\sum_{t=a}^{b-1}2L_{uw}[\tilde{y}](t)\eta^\sigma(t){_t}\Delta_b^\beta\eta(t)
=-\sum_{t=a}^{b-1}2L_{uw}[\tilde{y}](t)\eta^\sigma(t){_t}\Delta_{\rho(b)}^{-\nu}\Delta
\eta(t)=2L_{uw}[\tilde{y}](\tau)(\nu-1).$$ Moreover,
\begin{equation*}
\begin{split}
\sum_{t=a}^{b-1} & 2L_{vw}[\tilde{y}](t){_a}\Delta_t^\alpha\eta(t){_t}\Delta_b^\beta\eta(t)\\
&=-2\sum_{t=a}^{b-1}L_{vw}[\tilde{y}](t)\left\{\left(\Delta\eta(t)+\frac{\mu}{\Gamma(\mu+1)}
\cdot\sum_{s=a}^{t-1}(t+\mu-\sigma(s))^{(\mu-1)}\Delta\eta(s)\right)\right.\\
& \qquad \left. \cdot \left[\Delta\eta(t)
+\frac{\nu}{\Gamma(\nu+1)}\sum_{s=\sigma(t)}^{b-1}(s+\nu-\sigma(t))^{(\nu-1)}\Delta\eta(s)\right]\right\}\\
&=2(\nu-1)L_{vw}[\tilde{y}](\tau)+2(\mu-1)L_{vw}[\tilde{y}](\sigma(\tau)).
\end{split}
\end{equation*}
Finally, we have that
\begin{equation*}
\begin{split}
\sum_{t=a}^{b-1} & L_{ww}[\tilde{y}](t)({_t}\Delta_b^\beta\eta(t))^2\\
&=\sum_{t=a}^{\sigma(\tau)}L_{ww}[\tilde{y}](t)\left[\Delta\eta(t)
+\frac{\nu}{\Gamma(\nu+1)}\sum_{s=\sigma(t)}^{b-1}(s
+\nu-\sigma(t))^{(\nu-1)}\Delta\eta(s)\right]^2\\
&=\sum_{t=a}^{\tau-1}L_{ww}[\tilde{y}](t)\left[\frac{\nu}{\Gamma(\nu
+1)}\sum_{s=\sigma(t)}^{b-1}(s+\nu-\sigma(t))^{(\nu-1)}\Delta\eta(s)\right]^2\\
&\qquad +L_{ww}[\tilde{y}](\tau)(1-\nu)^2+L_{ww}[\tilde{y}](\sigma(\tau))\\
&=\sum_{t=a}^{\tau-1}L_{ww}[\tilde{y}](t)\left[\frac{\nu}{\Gamma(\nu+1)}\left\{(\tau
+\nu-\sigma(t))^{(\nu-1)}-(\sigma(\tau)+\nu-\sigma(t))^{(\nu-1)}\right\}\right]^2\\
&\qquad +L_{ww}[\tilde{y}](\tau)(1-\nu)^2+L_{ww}[\tilde{y}](\sigma(\tau)).
\end{split}
\end{equation*}
Similarly as we have done in \eqref{rui10}, we obtain that
$$\frac{\nu}{\Gamma(\nu+1)}\left[(\tau+\nu-\sigma(t))^{(\nu-1)}
-(\sigma(\tau)+\nu-\sigma(t))^{(\nu-1)}\right]
=\frac{\nu(1-\nu)\prod_{i=0}^{\tau-t-2}(\nu+i)}{(\sigma(\tau)-t)\Gamma(\sigma(\tau)-t)}.$$
We are done with the proof.
\end{proof}
A trivial corollary of our result
gives the discrete-time version of Legendre's necessary condition.
\begin{corollary}[The discrete-time Legendre condition
-- \textrm{cf.}, \textrm{e.g.}, \cite{CD:Bohner:2004,Zeidan2}]
\label{cor:2}
If $\tilde{y}$ is a solution to the problem
\eqref{eq:BPCV:DT}, then
$$L_{uu}[\tilde{y}](t)+2L_{uv}[\tilde{y}](t)+L_{vv}[\tilde{y}](t)
    +L_{vv}[\tilde{y}](\sigma(t))\geq 0$$
holds for all $t\in\mathbb{T}^{\kappa^2}$, where
$[\tilde{y}](t)=(t,\tilde{y}^{\sigma}(t),\Delta\tilde{y}(t))$.
\end{corollary}
\begin{proof}
We consider problem (P) with $\alpha=1$ and $L$ not depending on $w$.
The choice $\alpha=1$ implies $\mu=0$,
and the result follows immediately from Theorem~\ref{thm1}.
\end{proof}


\section{Examples}
\label{sec2}

In this section we present three illustrative examples.
The results were obtained using the open source
Computer Algebra System
\textsf{Maxima}.\footnote{\url{http://maxima.sourceforge.net}}
All computations were done running
\textsf{Maxima} on an Intel$\circledR$ Core$^{TM}$2 Duo, CPU of 2.27GHz with 3Gb of RAM.
Our \textsf{Maxima} definitions are given in Appendix.

\begin{example}
\label{ex:1}
Let us consider the following problem:
\begin{equation}
\label{eq:ex1} J_{\alpha}(y)=\sum_{t=0}^{b-1}
\left({_0}\Delta_t^\alpha y(t)\right)^2 \longrightarrow \min \, ,
\quad y(0) = A \, , \quad y(b) = B \, .
\end{equation}
In this case Theorem~\ref{thm1} is trivially satisfied. We obtain
the solution $\tilde{y}$ to our Euler-Lagrange equation \eqref{EL} for the case
$b = 2$ using the computer algebra system \textsf{Maxima}.
Using our \textsf{Maxima} package (see the definition of the
command \texttt{extremal} in Appendix) we do
\begin{verbatim}
         L1:v^2$
         extremal(L1,0,2,A,B,alpha,alpha);
\end{verbatim}
to obtain (2 seconds)
\begin{equation}
\label{eq:sol:ex1} \tilde{y}(1) = \frac{2\,\alpha\,B+\left(
{\alpha}^{3}-{\alpha}^{2}+2\,\alpha\right) \,A}{2\,{\alpha}^{2}+2}
\, .
\end{equation}
For the particular case $\alpha = 1$ the equality
\eqref{eq:sol:ex1} gives $\tilde{y}(1) = \frac{A+B}{2}$, which coincides
with the solution to the (non-fractional) discrete problem
\begin{equation*}
\sum_{t=0}^{1} \left(\Delta y(t)\right)^2
= \sum_{t=0}^{1} \left(y(t+1)-y(t)\right)^2 \longrightarrow \min
\, , \quad y(0) = A \, , \quad y(2) = B \, .
\end{equation*}
Similarly, we can obtain exact formulas of the extremal on
bigger intervals (for bigger values of $b$). For example, the
solution of problem \eqref{eq:ex1} with $b = 3$ is (35 seconds)
\begin{equation*}
\begin{split}
\tilde{y}(1) &= \frac{\left( 6\,{\alpha}^{2}+6\,\alpha\right) \,B+\left( 2\,{\alpha}^{5}
+2\,{\alpha}^{4}+10\,{\alpha}^{3}-2\,{\alpha}^{2}
+12\,\alpha\right) \,A}{3\,{\alpha}^{4}+6\,{\alpha}^{3}
+15\,{\alpha}^{2}+12} \, ,\\
\tilde{y}(2) &= \frac{\left(
12\,{\alpha}^{3}+12\,{\alpha}^{2}+24\,\alpha\right) \,B+\left(
{\alpha}^{6}+{\alpha}^{5}+7\,{\alpha}^{4}-{\alpha}^{3}+4\,{\alpha}^{2}+12\,\alpha\right)
\,A}{6\,{\alpha}^{4}+12\,{\alpha}^{3}+30\,{\alpha}^{2}+24} \, ;
\end{split}
\end{equation*}
and the solution of problem \eqref{eq:ex1} with $b = 4$ is (72
seconds)
\begin{equation*}
\begin{split}
\tilde{y}(1) &=
\frac{3\,{\alpha}^{7}+15\,{\alpha}^{6}+57\,{\alpha}^{5}+69\,{\alpha}^{4}
+156\,{\alpha}^{3}-12\,{\alpha}^{2}+144\,\alpha}{\xi} A\\
&\qquad + \frac{24\,{\alpha}^{3}+72\,{\alpha}^{2}+48\,\alpha}{\xi} B \, ,\\
\tilde{y}(2) &=
\frac{{\alpha}^{8}+5\,{\alpha}^{7}+22\,{\alpha}^{6}+32\,{\alpha}^{5}
+67\,{\alpha}^{4}+35\,{\alpha}^{3}+54\,{\alpha}^{2}+72\,\alpha}{\xi} A\\
&\qquad + \frac{24\,{\alpha}^{4}+72\,{\alpha}^{3}+120\,{\alpha}^{2}
+72\,\alpha}{\xi} B \, ,\\
\tilde{y}(3) &=
\frac{{\alpha}^{9}+6\,{\alpha}^{8}+30\,{\alpha}^{7}+60\,{\alpha}^{6}+117\,{\alpha}^{5}
+150\,{\alpha}^{4}-4\,{\alpha}^{3}+216\,{\alpha}^{2}+288\,\alpha}{\zeta} A\\
&\qquad + \frac{72\,{\alpha}^{5}+288\,{\alpha}^{4}+792\,{\alpha}^{3}
+576\,{\alpha}^{2}+864\,\alpha}{\zeta} B \, ,
\end{split}
\end{equation*}
where
\begin{equation*}
\begin{split}
\xi &= 4\,{\alpha}^{6} +24\,{\alpha}^{5}+88\,{\alpha}^{4}+120\,{\alpha}^{3}
+196\,{\alpha}^{2}+144 \, , \\
\zeta &= 24\,{\alpha}^{6}+144\,{\alpha}^{5}+528\,{\alpha}^{4}
+720\,{\alpha}^{3}+1176\,{\alpha}^{2}+864 \, .
\end{split}
\end{equation*}
Consider now problem \eqref{eq:ex1} with $b = 4$, $A=0$, and $B=1$.
In Table~\ref{tab:1} we show the extremal values
$\tilde{y}(1)$, $\tilde{y}(2)$,
$\tilde{y}(3)$, and corresponding $\tilde{J}_{\alpha}$, for some values of
$\alpha$. Our numerical results show that the fractional extremal
converges to the classical (integer order) extremal
when $\alpha$ tends to one. This is illustrated
in Figure~\ref{Fig:0}. The numerical results from Table~\ref{tab:1}
and Figure~\ref{Fig:2} show that for this problem the smallest value of
$\tilde{J}_{\alpha}$, $\alpha\in]0,1]$, occur for $\alpha=1$
(\textrm{i.e.}, the smallest value of
$\tilde{J}_{\alpha}$ occurs for the classical non-fractional case).
\begin{figure}[htp]
\begin{center}
\includegraphics[scale=0.6]{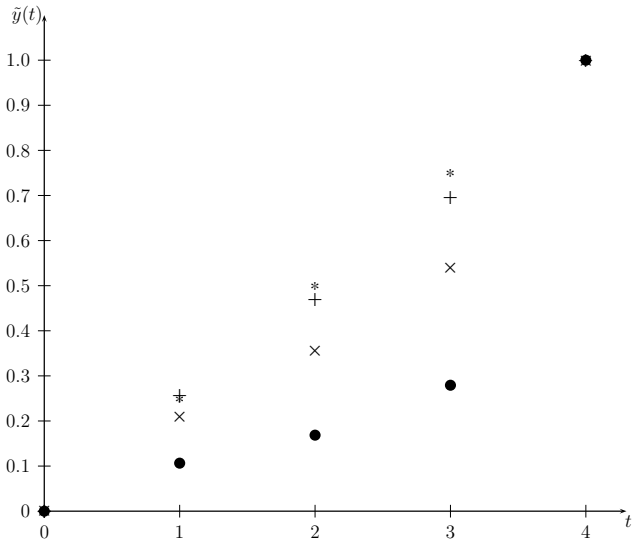}
  \caption{Extremal $\tilde{y}(t)$ of Example~\ref{ex:1}
  with $b = 4$, $A=0$, $B=1$, and different $\alpha$'s
  ($\bullet$: $\alpha=0.25$; $\times$: $\alpha=0.5$;
  $+$: $\alpha=0.75$; $\ast$: $\alpha=1$).}\label{Fig:0}
\end{center}
\end{figure}
{\small
\begin{table}[!htbp]
  \centering
  \begin{tabular}{|c|c|c|c|c|}
    \hline
    $\alpha$ & $\tilde{y}(1)$ & $\tilde{y}(2)$ & $\tilde{y}(3)$ & $\tilde{J}_{\alpha}$\\
    \hline
    0.25 & 0.10647146897355 & 0.16857982587479 & 0.2792657904952 & 0.90855653524095 \\
    \hline
    0.50 & 0.20997375328084 & 0.35695538057743 & 0.54068241469816 & 0.67191601049869 \\
    \hline
    0.75 & 0.25543605027861 & 0.4702345471038 & 0.69508876506414 & 0.4246209666969\\
    \hline
    1 & 0.25 & 0.5 & 0.75 & 0.25\\
    \hline
  \end{tabular}
\caption{The extremal values $\tilde{y}(1)$, $\tilde{y}(2)$ and $\tilde{y}(3)$
of problem \eqref{eq:ex1} with $b = 4$, $A=0$, and $B=1$
for different $\alpha$'s.}\label{tab:1}
\end{table}
}
\begin{figure}[htp]
\begin{center}
\includegraphics[scale=0.6]{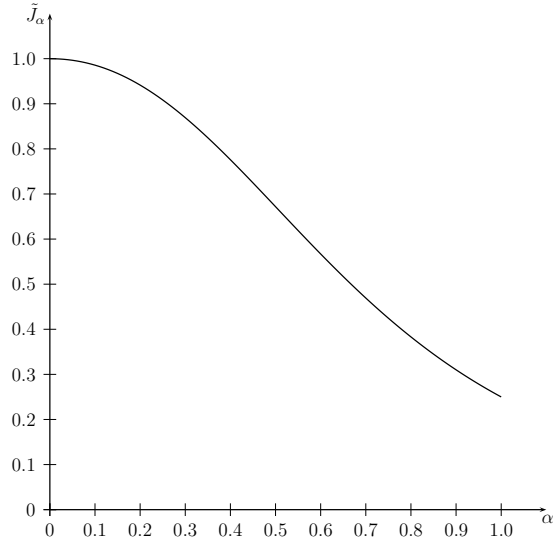}
\caption{Function $\tilde{J}_{\alpha}$ of Example~\ref{ex:1}
with $b = 4$, $A=0$, and $B=1$.}\label{Fig:2}
\end{center}
\end{figure}
\end{example}

\begin{example}
\label{ex:2}
In this example we generalize problem
\eqref{eq:ex1} to
\begin{equation}
\label{eq:ex2}
 J_{\alpha,\beta}=\sum_{t=0}^{b-1} \gamma_1
\Bigl({_0}\Delta_t^\alpha y(t)\Bigr)^2 + \gamma_2
\Bigl({_t}\Delta_b^\beta y(t)\Bigr)^2 \longrightarrow \min \, ,
\quad y(0) = A \, , \quad y(b) = B \, .
\end{equation}
As before, we solve the associated Euler-Lagrange equation
\eqref{EL} for the case $b = 2$ with the help of our \textsf{Maxima} package
(35 seconds):
\begin{verbatim}
         L2:(gamma[1])*v^2+(gamma[2])*w^2$
         extremal(L2,0,2,A,B,alpha,beta);
\end{verbatim}
\begin{equation*}
\tilde{y}(1) =  \frac{\left(2\,\gamma_{2}\,\beta+\gamma_{1}\,{\alpha}^{3}-\gamma_{1}\,{\alpha}^{2}
+2\,\gamma_{1}\,\alpha\right) \,A+\left(\gamma_{2}\,{\beta}^{3}-\gamma_{2}\,{\beta}^{2}
+2\,\gamma_{2}\,\beta+2\,\gamma_{1}\,\alpha\right) \,B}{2\,\gamma_{2}\,{\beta}^{2}+2\,
\gamma_{1}\,{\alpha}^{2}+2\,\gamma_{2}+2\,\gamma_{1}} \, .
\end{equation*}
Consider now problem \eqref{eq:ex2} with
$\gamma_1=\gamma_2=1$, $b=2$, $A=0$, $B=1$, and $\beta=\alpha$.
In Table~\ref{tab:2} we show the values of $\tilde{y}(1)$ and
$\tilde{J}_{\alpha} := J_{\alpha,\alpha}(\tilde{y}(1))$ for some values of $\alpha$.
We concluded, numerically, that the fractional extremal $\tilde{y}(1)$ tends to
the classical (non-fractional) extremal when $\alpha$ tends to one.
Differently from Example~\ref{ex:1},
the smallest value of $\tilde{J}_{\alpha}$, $\alpha\in]0,1]$, does not
occur here for $\alpha=1$ (see Figure~\ref{Fig:4}).
The smallest value of $\tilde{J}_{\alpha}$,
$\alpha\in]0,1]$, occurs for $\alpha=0.61747447161482$.
{\small
\begin{table}[!htbp]
  \centering
  \begin{tabular}{|c|c|c|}
    \hline
    $\alpha$ & $\tilde{y}(1)$  & $\tilde{J}_{\alpha}$\\
    \hline
    0.25 & 0.22426470588235 &  0.96441291360294 \\
    \hline
    0.50 & 0.375 & 0.9140625 \\
    \hline
    0.75 & 0.4575 &  0.91720703125\\
    \hline
    1 & 0.5 & 1\\
    \hline
  \end{tabular}
  \caption{The extremal $\tilde{y}(1)$ of problem \eqref{eq:ex2}
  for different values of $\alpha$
  ($\gamma_1=\gamma_2=1$, $b=2$, $A=0$, $B=1$, and $\beta = \alpha$).}\label{tab:2}
\end{table}
}
\begin{figure}[htp]
\begin{center}
\includegraphics[scale=0.6]{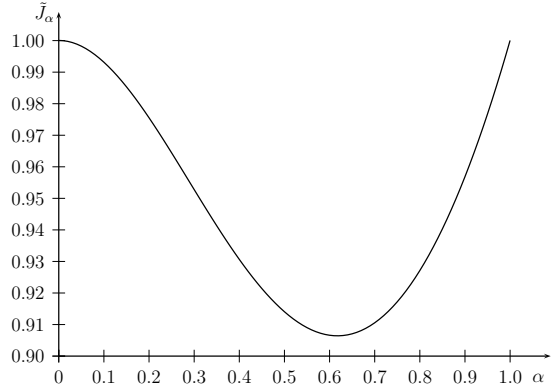}
  \caption{Function $\tilde{J}_{\alpha}$ of Example~\ref{ex:2}
with $\gamma_1=\gamma_2=1$, $b=2$, $A=0$, $B=1$, and $\beta = \alpha$.}\label{Fig:4}
\end{center}
  \end{figure}
\end{example}

\begin{example}
\label{ex:3}
Our last example is a discrete version
of the fractional continuous problem
\cite[Example~2]{agr2}:
\begin{equation}
\label{eq:ex3} J_{\alpha}=\sum_{t=0}^{1}
\frac{1}{2}\left({_0}\Delta_t^\alpha y(t)\right)^2-y^{\sigma}(t)
\longrightarrow \min \, , \quad y(0) = 0 \, , \quad y(2) = 0 \, .
\end{equation}
The Euler-Lagrange extremal of \eqref{eq:ex3}
is easily obtained with our \textsf{Maxima} package (4 seconds):
\begin{verbatim}
         L3:(1/2)*v^2-u;$
         extremal(L3,0,2,0,0,alpha,beta);
\end{verbatim}
\begin{equation}
\label{eq:sol:ex3} \tilde{y}(1) =  \frac{1}{{\alpha}^{2}+1} \, .
\end{equation}
For the particular case $\alpha = 1$ the equality \eqref{eq:sol:ex3}
gives $\tilde{y}(1) = \frac{1}{2}$, which coincides with the solution
to the non-fractional discrete problem
\begin{gather*}
\sum_{t=0}^{1} \frac{1}{2}\left(\Delta y(t)\right)^2-y^{\sigma}(t) =
\sum_{t=0}^{1} \frac{1}{2}\left(y(t+1)-y(t)\right)^2-y(t+1)
\longrightarrow \min \, , \\
y(0) = 0 \, , \quad y(2) = 0 \, .
\end{gather*}
In Table~\ref{tab:3} we show the values of $\tilde{y}(1)$ and
$\tilde{J}_{\alpha}$ for some $\alpha$'s.
As seen in Figure~\ref{Fig:6}, for $\alpha=1$
one gets the maximum value of $\tilde{J}_{\alpha}$, $\alpha\in]0,1]$.
{\small
\begin{table}[!htbp]
  \centering
  \begin{tabular}{|c|c|c|}
    \hline
    $\alpha$ & $\tilde{y}(1)$  & $\tilde{J}_{\alpha}$\\
    \hline
    0.25 & 0.94117647058824 &  -0.47058823529412 \\
    \hline
    0.50 & 0.8 & -0.4 \\
    \hline
    0.75 & 0.64 &  -0.32\\
    \hline
    1 & 0.5 & -0.25\\
    \hline
  \end{tabular}
  \caption{Extremal values $\tilde{y}(1)$ of \eqref{eq:ex3}
  for different $\alpha$'s}\label{tab:3}
\end{table}
}
\begin{figure}[htp]
\begin{center}
\includegraphics[scale=0.6]{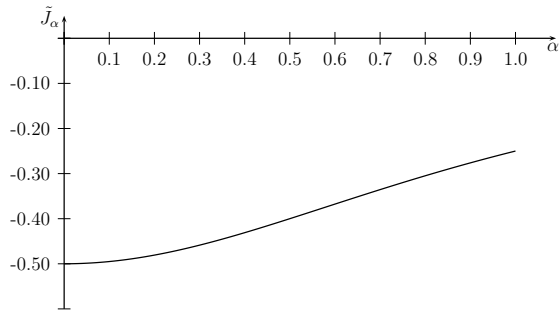}
\caption{Function $\tilde{J}_{\alpha}$ of Example~\ref{ex:3}.}\label{Fig:6}
\end{center}
\end{figure}
\end{example}


\section{Conclusion}
\label{sec:conc}

The discrete-time calculus is a very important tool in
practical applications and in the modeling of real phenomena.
Therefore, it is not a surprise that fractional discrete calculus
is recently under strong development. Possible areas of application
include the signal processing, where fractional derivatives
of a discrete-time signal are particularly useful
to describe noise processes \cite{Ortig:B}.

In this paper we introduce the study of fractional discrete-time
problems of the calculus of variations or order $\alpha$,
$0 < \alpha \le 1$, with left and right discrete operators
of Riemann--Liouville type. For $\alpha = 1$
we obtain the classical discrete-time results
of the calculus of variations \cite{book:DCV}.
Main results of the paper include
a fractional summation by parts formula (Theorem~\ref{teor1}),
a fractional discrete-time Euler--Lagrange equation (Theorem~\ref{thm0}),
transversality conditions \eqref{rui1} and \eqref{rui2},
and a fractional discrete-time Legendre condition (Theorem~\ref{thm1}).
From the analysis of the results obtained from computer experiments,
we conclude that when the value of $\alpha$ approaches one,
the optimal value of the fractional discrete functional converges
to the optimal value of the classical (non-fractional) discrete problem.
On the other hand, the value of $\alpha$ for which the functional
attains its minimum varies with the concrete problem under consideration.

This research is in its beginning phase, and it will be developed further in the future.
Indeed, being the first work on fractional difference variational problems,
much remains to be done. For example, one can extend the present results for higher-order problems
of the calculus of variations with fractional discrete derivatives
of any order. Moreover, our work also opens new possibilities of research for
fractional continuous variational problems. In particular, to prove
a fractional continuous Legendre necessary optimality condition,
analogous to the fractional discrete result given by Theorem~\ref{thm1},
is a stimulating open question.
One of the referees called our attention to the fact that
the initial conditions considered in this work have ARMA formats.
The problem of finding a general formulation
leading to the specification of the ARMA parameters
seems to be also an interesting question.


\section*{Appendix}

The following \textsf{Maxima} code implements Theorem~\ref{thm0}. Examples
illustrating the use of our procedure \texttt{extremal} are found in
Section~\ref{sec2}.

\small
\begin{verbatim}

kill(all)$
ratprint:false$
simpsum:true$
tlimswitch:true$

sigma(t):=t+1$

rho(t):=t-1$

rho2(t):=rho(rho(t))$

Delta(exp,t):=block( define(f12(t),exp),
return((f12(sigma(t))-f12(t))) )$

p(x,y):=(gamma(x+1))/(gamma(x+1-y))$

SumL(a,t,nu,exp):=block(
define(f1(x),exp),
f1(t)+nu/gamma(nu+1)*sum((p(t+nu-sigma(r),nu-1))*f1(r),r,(a),((t-1)))
)$

SumR(t1,b,nu1,exp1):=block(
define(f2(x),exp1),
f2(t1)+nu1/gamma(nu1+1)*sum((p(s+nu1-sigma(t1),nu1-1))*f2(s),s,(t1+1),b)
)$

DeltaL(a2,t2,alpha2,exp2):=block(
[alpha1:ratsimp(alpha2),a:ratsimp(a2),t:ratsimp(t2)],
define(f3(x),exp2), define(q(x),SumL(a,x,1-alpha1,f3(x))),
q0:q(o),
o4:float(ev(ratsimp(Delta(q0,o)),nouns)), o41:subst(o=t,o4),
remfunction(f), remfunction(q), return(o41)
)$

DeltaR(t3,b,alpha3,exp3):=block(
[alpha1:ratsimp(alpha3),b:ratsimp(b),t:ratsimp(t3)],
define(f4(x),exp3), define(q1(o),(SumR(x,b,1-alpha1,f4(x)))),
q10:q1(z),
o5:float(ev(ratsimp(-Delta((SumR(x,b,1-alpha1,f4(x))),x)),nouns)),
o51:subst(x=t,o5), remfunction(f), remfunction(q1), return(o51)
)$

EL(exp7,a,b,alpha7,beta7):=block(
[a:ratsimp(a),b:ratsimp(b),alpha:ratsimp(alpha7),beta:ratsimp(beta7)],
define(LL(t,u,v,w),exp7), b1:diff(LL(t,u,v,w),u),
sa:subst([u=y(sigma(t)),v=DeltaL(a,t,alpha,y(o)),
                        w=DeltaR(t,b,beta,y(o))],b1),
b2:diff(LL(t,u,v,w),v),
sb:subst([t=x,u=y(sigma(x)),v=DeltaL(a,x,alpha,y(x)),
                            w=DeltaR(x,b,beta,y(x))],b2),
sb1:DeltaR(o,rho(b),alpha,sb), sb11:subst(o=t,sb1),
b3:diff(LL(t,u,v,w),w),
sc:subst([t=x,u=y(sigma(x)),v=DeltaL(a,x,alpha,y(x)),
                            w=DeltaR(x,b,beta,y(x))],b3),
sc2:DeltaL(a,p2,beta,sc), sc22:subst(p2=t,sc2),
return(sa+sb11+sc22)
)$

ELt(exp8,a,b,alpha8,beta8,t8):=
ratsimp(subst(t=t8,EL(exp8,a,b,alpha8,beta8)))$

extremal(L,a,b,A9,B9,alpha9,beta9):=block(
[a:ratsimp(a),b:ratsimp(b),alpha:ratsimp(alpha9),
beta:ratsimp(beta9),A1:ratsimp(A9), B1:ratsimp(B9)],
eqs:makelist(ratsimp(ELt(L,a,b,alpha,beta,a+i)),i,0,ratsimp((rho2(b)-a))),
vars:makelist(y(ratsimp(a+i)),i,1,ratsimp((rho(b)-a))),
Xi:[a],Xf:[b],Yi:[A1],Yf:[B1],
X:makelist(ratsimp(i),i,1,ratsimp((rho(b)-a))), X:append(Xi,X,Xf),
sols:algsys(subst([y(a)=A1,y(b)=B1],eqs),vars),
Y:makelist(rhs(sols[1][i]),i,1,ratsimp((rho(b)-a))),
Y:append(Yi,Y,Yf),
return(makegamma(ratsimp(minfactorial(makefact(sols[1]))))
)$
\end{verbatim}

\normalsize


\section*{Acknowledgements}

This work is part of the first author's PhD project,
which is carried out at the University of Aveiro
under the Doctoral Programme \emph{Mathematics and Applications}
of Universities of Aveiro and Minho. The financial support
of the Polytechnic Institute of Viseu and
\emph{The Portuguese Foundation for Science and Technology} (FCT),
through the ``Programa de apoio \`{a} forma\c{c}\~{a}o avan\c{c}ada
de docentes do Ensino Superior Polit\'{e}cnico'',
PhD fellowship SFRH/PROTEC/49730/2009, is here gratefully acknowledged.
The second author was supported by FCT through the PhD fellowship SFRH/BD/39816/2007;
the third  author was supported by FCT through the
R\&D unit CEOC and the \emph{Center of Research and Development
in Mathematics and Applications} (CIDMA).

We are grateful to three anonymous referees
for several relevant and stimulating remarks,
contributing to improve the quality of the paper.



\end{document}